\documentclass[reqno, 12pt]{amsart}

\usepackage{amsmath,graphicx}

\usepackage[algoruled,lined,noresetcount,norelsize]{algorithm2e}

\usepackage{amssymb,amsthm}
\usepackage{mathrsfs}
\usepackage{mathabx}
\usepackage[T1]{fontenc}
\usepackage{lmodern}
\usepackage{ulem}
\usepackage[babel]{microtype}
\usepackage[english]{babel}
\usepackage{geometry}
\usepackage{enumitem}
\usepackage{xcolor}
\usepackage[colorlinks,bookmarks,linkcolor=purple,citecolor=black]{hyperref}

\usepackage{tikz}
\usetikzlibrary{patterns,patterns.meta}

\definecolor{xred}{RGB}{240,0,30}
\definecolor{xblue}{RGB}{0,119,187}
\definecolor{xgreen}{RGB}{55,150,5}
\definecolor{xgray}{RGB}{200,200,200}

\def\itm#1{\rm ({#1})} 
\def\itmit#1{\itm{\it #1\,}} 
\def\rom{\itmit{\roman{*}}} 
\def\abc{\itmit{\alph{*}}}

\newtheorem{theorem}{Theorem}
\newtheorem{conjecture}[theorem]{Conjecture}
\newtheorem{proposition}[theorem]{Proposition}
\newtheorem{corollary}[theorem]{Corollary}
\newtheorem{lemma}[theorem]{Lemma}

\newtheorem{problem}[theorem]{Problem}

\newtheorem{remark}[theorem]{Remark}

\newtheorem{definition}[theorem]{Definition}

\title{On Ramsey-type problems for paths and cycles with few colour changes}

\author[P. Allen]{Peter Allen}
\author[J. Böttcher]{Julia Böttcher}
\author[D. Clemens]{Dennis Clemens}
\author[F. Hamann]{Fabian Hamann}
\author[J. Skokan]{Jozef Skokan}
\author[A. Taraz]{Anusch Taraz}

\address{(PA|JB|JS) London School of Economics, Department of Mathematics, Houghton Street, London WC2A 2AE, UK}
\email{p.d.allen|j.boettcher|j.skokan@lse.ac.uk}
\address{(DC|FH|AT) Technische Universität Hamburg, Institut f\"ur Mathematik, Am Schwarzenberg-Campus 3, 21073 Hamburg, Germany }
\email{dennis.clemens|fabian.hamann|taraz@tuhh.de}

\begin{document}

%%%%%%%%%%%%%%%%%%%%%%%%%%%%%%%%
% Abstract
%%%%%%%%%%%%%%%%%%%%%%%%%%%%%%%%
\begin{abstract}
    In 1967, Gerencser and Gy\'arf\'as determined the exact values of the two-colour Ramsey numbers of paths. In a footnote, they made the following observation: Every $2$-edge-coloured complete graph contains a Hamilton path with at most one colour change. Later, this led to a challenging and still wide open conjecture about covering edge-coloured complete graphs with monochromatic paths. Inspired by the original statement, we study paths and cycles with few colour changes in $3$-edge-coloured complete graphs. For this, we introduce a new Ramsey-type parameter: For $q,k \in \mathbb{N}$ and a graph $G$, let $R_q^k(G)$ denote the smallest $N \in \mathbb{N}$ such that every $q$-edge-coloured complete graph on $N$ vertices contains a copy of $G$ with at most $k$ vertices that are incident to edges in $G$ of different colours. For paths, we show that $R_3^1(P_n) = \frac{3n}{2} + O(1)$, and for even cycles, we show that $R_3^2(C_n) = \frac{3n}{2} + o(n)$.
\end{abstract}

\maketitle

%%%%%%%%%%%%%%%%%%%%%%%%%%%%%%%%
% Introduction
%%%%%%%%%%%%%%%%%%%%%%%%%%%%%%%%
\section{Introduction}
The \textit{Ramsey number} $R(G)$ of a graph $G$ is defined as the minimum $N \in \mathbb{N}$ such that any $2$-edge-colouring of $K_N$ contains a monochromatic copy of $G$. If we instead consider $q$-edge-colourings for some $q \in \mathbb{N}$, we write $R_q(G)$. A fundamental result of Ramsey \cite{ramsey1930problem} shows that $R_q(G)$ exists for all $G$ and $q \in \mathbb{N}$, but providing precise values or even close bounds proves to be challenging. In particular, the original bounds on $R(K_n)$ for $n \in \mathbb{N}$, shown by Erd\H{o}s and Szekeres in 1935 \cite{erdos1935combinatorial} and by Erd\H{o}s in 1947 \cite{erdos1947some} show that the growth of $R(K_n)$ is exponential in $n$, but the bounds are also exponentially far apart. Even though the problem has been well studied since then, no exponential improvement to the bounds had been achieved until a recent breakthrough by Campos, Griffiths, Morris, and Sahasrabudhe \cite{campos2023exponential}; see also~\cite{balister2024upper,gupta2024optimizing}
for further improvements and generalizations.

If $G$ is a sparse graph, some more precise results are known. For graphs $G$ with maximum degree $\Delta \in \mathbb{N}$, the Ramsey number $R(G)$ actually grows linearly with the number of vertices of $G$. This was originally conjectured by Burr and Erd\H{o}s \cite{BurrErdos1975}, and proved by Chv\'atal, Rödl, Szemer\'edi and Trotter \cite{chvatal1983ramsey} using Szemer\'edi's regularity lemma (see also \cite{Lee_BurrErdos} for a more general result on $d$-degenerate graphs).

Among bounded degree graphs, the Ramsey numbers of paths and cycles have been studied extensively. While $R_2(P_n) = \lfloor \frac{3n}{2} \rfloor - 1$ had already been determined by Gerencs\'er and Gy\'arf\'as in 1967 \cite{gerencser1967ramsey}, it took until 2007 to determine the precise value of $R_3(P_n)$ for large enough values of~$n$, using the regularity lemma \cite{gyarfas2007three}.
\begin{theorem}[Theorem 1 in \cite{gyarfas2007three}]
    For large enough $n$, we have
    \[R_3(P_n) =
\begin{cases}
  2n - 2, & \text{if $n$ is even}, \\
  2n - 1, & \text{if $n$ is odd}.
\end{cases}
\]
\end{theorem}

For a general number of colours $q \in \mathbb{N}$, the value of $R_q(P_n)$ is still not known, not even asymptotically. The best known bounds are due to \cite{knierim2019improved} and \cite{yongqi2006new}, which show that
$$(q - 1 + o(1))n \leq R_q(P_n) \leq (q - \frac12 + o(1))n.$$

For the case of cycles, odd and even cycles behave very differently. While even cycles behave similarly to paths, the Ramsey number of odd cycles is significantly higher. To be more precise, for the $2$-coloured case we have \[
R_2(C_n) =
\begin{cases}
  6, &\text{if $3 \leq n \leq 4$},\\
  \frac{3n}{2} - 1, & \text{if $n$ is even}, \\
  2n - 1, & \text{if $n$ is odd}
\end{cases}
\] due to~\cite{bondy1973ramsey,faudree1974all,rosta1973ramsey}. Again, for the $3$-coloured case, precise values are only known for large enough $n$, since the proofs rely on the regularity lemma. The case of odd cycles was solved in~\cite{kohayakawa20053}, while a proof for even cycles appeared shortly after in~\cite{benevides20093}.
\begin{theorem}[Theorem~1 in \cite{benevides20093} and Theorem~1.1 in \cite{kohayakawa20053}]\label{thm:ramsey_evencycle}
For $n$ large enough, we have
\[
R_3(C_n) =
\begin{cases}
  2n, & \text{if $n$ is even}, \\
  4n - 3, & \text{if $n$ is odd}.
\end{cases}
\]
\end{theorem}

A related problem originated from the study of $R_2(P_n)$ by Gerencs\'er and Gy\'arf\'as \cite{gerencser1967ramsey}. In a footnote, they mentioned the following result, for which a simple proof can be found e.g. in \cite{gyarfas1983vertex}.
\begin{proposition}[Gerencs\'er, Gy\'arf\'as~\cite{gerencser1967ramsey}]\label{prop:2col.path.cover}
    For $n \in \mathbb{N}$, let the edges of $K_n$ be coloured with two colours. Then there is a Hamilton path with at most one colour change (i.e.~with at most one vertex incident to edges of different colours in the Hamilton path).\footnote{
      Gerencs\'er and Gy\'arf\'as did not use the terminology of colour changes, but called such a path simple.
    }   
\end{proposition}

Later, Gy\'arf\'as posed the following conjecture, for which the case $q=2$ already follows from Proposition~\ref{prop:2col.path.cover}.
\begin{conjecture}[Gy\'arf\'as~\cite{gyarfas1989covering}]\label{conj:path.cover}
    Let $n,q \in \mathbb{N}$, and let the edges of $K_n$ be coloured with $q$ colours. Then, the vertices of $K_n$ can be covered by $q$ monochromatic paths.
\end{conjecture}

A slightly stronger version of Conjecture~\ref{conj:path.cover} was proposed by Erd\H{o}s, Gy\'arf\'as, and Pyber \cite{erdHos1991vertex}. They conjectured that it might even be possible to cover the vertices by $q$ monochromatic cycles. This is true for $q = 2$, as proved by Bessy and Thomass\'e~\cite{BessyThomasse} (it was previously proved for sufficiently large $n$, see~\cite{ALLEN_2008,luczak1998partitioning}). However, Pokrovskiy~\cite{pokrovskiy2014partitioning} constructed examples showing that the conjecture is actually false for $q \geq 3$. Still, Conjecture~\ref{conj:path.cover} remains open apart from the case $q = 3$, which was also proved by Pokrovskiy~\cite{pokrovskiy2014partitioning}.

\begin{theorem}[Theorem 3.2 and Theorem 3.3 in \cite{pokrovskiy2014partitioning}]\label{thm:path-cover}
    For $n \in \mathbb{N}$, let the edges of $K_n$ be coloured with three colours. Then there is a vertex-partition of $K_n$ into three monochromatic paths, of which at most two have the same colour, where a single vertex and the empty set is also considered to be a path.
\end{theorem}

If we focus on the original formulation of paths with colour changes rather than partitioning into monochromatic structures, not so much is known yet. Raynaud~\cite{raynaud1973circuit} proved the statement of Proposition~\ref{prop:2col.path.cover} for $2$-coloured complete oriented graphs. The structure of Hamilton cycles with two colour changes in $2$-edge-coloured complete graphs was further investigated in \cite{bialostocki1991simple}, where the length of the monochromatic path segments was considered. This was built upon in \cite{manoussakis1996cycles}, which studied the problem of finding $(s,t)$-cycles in $2$-edge-coloured complete graphs, where \textit{$(s,t)$-cycles} are defined as cycles of length $s + t$ in which $s$ consecutive edges are of one colour and the remaining $t$ edges are of the other colour.

Another related problem is the conjecture of Feder and Subi \cite{feder2013hypercube} about $2$-edge-coloured hypercubes. They conjecture that for all $d \in \mathbb{N}$, every $2$-edge-coloured $d$-dimensional hypercube $Q_d$ contains two vertices at distance $d$ which are connected by a path with at most one colour change (see also \cite{leader2014long,NorineEdgeAntipodal} for related conjectures). The conjecture of Feder and Subi is still wide open, as the best known upper bound on the number of colour changes is $(\tfrac{3}{8} + o(1))d$ due to Dvo\v{r}\'ak \cite{dvovrak2020note}.

Motivated by these results and open problems, we study a new Ramsey variant where we are not required to find a monochromatic subgraph, but allow for a few colour changes. While the notion of colour changes seems intuitive for paths and cycles, we want to make this notion more precise and also applicable to other kinds of graphs with the following definition.

\begin{definition}[Colour changes]
    Let $G = (V,E)$ be a graph with an edge colouring $f : E \rightarrow \mathbb{N}$. Let \[S := \big\{v \in V : \exists vx, vy \in E \text{ with } f(vx) \neq f(vy)\big\}\] be the set of vertices which are adjacent to edges of different colours. We say that $G$ has $k := |S|$ \textup{colour changes}. For a vertex $v \in S$, we say that a \textup{colour change occurs at $v$}.
\end{definition}

We are interested in the following generalisation of the Ramsey number of graphs.

\begin{definition}[Ramsey numbers with colour changes]
    Let $q,k \in \mathbb{N}$, and let $G$ be a graph. We define $R_q^k(G)$ to be the smallest integer $N \in \mathbb{N}$ such that every $q$-edge-colouring of $K_N$ contains a copy of $G$ with at most $k$ colour changes.
\end{definition}

Note that when~$G$ is a connected graph, then $R_q^0(G) = R_q(G)$ is the usual Ramsey number.
Further, if for a graph $G$ on $n$ vertices we have $R_q^k(G) = n$, this also implies a corresponding partitioning result, i.e., for every $q$-edge-colouring of $K_n$, its vertices can be partitioned into $k + 1$ monochromatic subgraphs of $G$.
In particular, $R_q^k(P_n)$ is closely related to Conjecture~\ref{conj:path.cover} which we will discuss in more detail in the concluding remarks. 

In this paper, we will focus on $R_q^1(P_n)$ and $R_q^2(C_n)$ (note that a cycle with one colour change does not exist). 
For $q = 2$, Proposition~\ref{prop:2col.path.cover} already gives us $R_2^1(P_n) = n$ and $R_2^2(C_n) = n$. We will consider the next open case, $q = 3$. For paths, we obtain bounds on $R_3^1(P_n)$ which are at most $5$ apart. Interestingly, $R_3^1(P_n)$ has the same asymptotic growth as $R_2(P_n)$, i.e. the extra colour is exactly offset by the possibility of using a colour change.

\begin{theorem}\label{thm:main_path}
For every large enough $n\in \mathbb{N}$, we have
\[  \left \lfloor \frac{3n-2}{2} \right\rfloor  \leq   R_3^1(P_n)\leq 3 \left \lceil \frac{n}{2} \right \rceil + 2\, \,,\]
and if $4 \mid n$, then moreover $R_3^1(P_n) \leq \frac{3n}{2}$.
\end{theorem}

Similarly for cycles, we show that $R_3^2(C_n)$ for even $n$ asymptotically behaves like $\frac{3n}{2}$.

\begin{theorem}\label{thm:main_even.cycle}
For every large enough even $n\in \mathbb{N}$, we have
$$ R_3^2(C_n) = \left(\frac{3}{2} + o(1) \right)n .$$
\end{theorem}

In fact, almost all parts of our proof work for odd and even $n$. For more discussion on odd cycles, see the concluding remarks.

We remark that our proofs do not employ the regularity method, but are dependent on Theorem~\ref{thm:ramsey_evencycle} which does make use of it.

%%%%%%%%%%%%%%%%%%%%%%%%%%%%%%%%
% Orginization
%%%%%%%%%%%%%%%%%%%%%%%%%%%%%%%%
\bigskip \noindent \textbf{Organization.}
In Section~\ref{section:prelim}, we collect some useful results and prove some easy corollaries. Afterwards, in 
Section~\ref{section:lowerbound} we present a colouring of $K_N$ for $N = \lfloor \frac{3n-4}{2} \rfloor$, which does not contain a path of length $n$ with at most one colour change, thus showing the required lower bounds for Theorem~\ref{thm:main_path} and Theorem~\ref{thm:main_even.cycle}. Sections~\ref{section:upper.bound.path} and~\ref{section:upper.bound.cycle} contain the proofs of the upper bounds for Theorem~\ref{thm:main_path} and Theorem~\ref{thm:main_even.cycle}, respectively. In the last section, we discuss further consequences and limitations of our results and state some open problems.

%%%%%%%%%%%%%%%%%%%%%%%%%%%%%%%%
% Notation
%%%%%%%%%%%%%%%%%%%%%%%%%%%%%%%%
\bigskip \noindent \textbf{Notation.}
Our notation is mostly standard and follows~\cite{west2001introduction}. For $n \in \mathbb{N}$, we write $[n] := \{i \in \mathbb{N} : i \leq n \}.$
For a graph $G$, we denote its vertex set by $V(G)$ and its edge set by $E(G)$. For an edge $\{v,w\} \in E(G)$, we write $vw$ for brevity. The \textit{minimum degree} of $G$ is denoted by $\delta(G) := \min\{\deg_G(v) : v \in V(G)\}.$ For $A,B \subset V(G)$, we write $E(A) := \{ e \in E(G) : e \subset A\}$ and $E(A,B) := \{ab \in E(G) : a \in A, b \in B\}$. Further, the subgraph induced by $A$ is denoted by $G[A] := (A,E(A))$. The \textit{neighbourhood} of $A$ is defined as $N(A) := \{v \in V(G) \setminus A : \exists u \in A, uv \in E(G)\},$ and the \textit{distance} between vertices $v,w \in V(G)$ is denoted by $\mathrm{dist}_G(v,w)$, i.e., the length of a shortest path from $v$ to $w$ in $G$.

The graph $P_n$ denotes the path on $n$ vertices, including $P_1$, which is a single vertex, and $C_n$ denotes the cycle on $n$ vertices. We represent a path on $n$ vertices, i.e.~a copy of $P_n$, in some graph $G$ by a sequence of distinct vertices $(p_1,p_2,\ldots,p_n)$ with $p_ip_{i+1} \in E(G)$ for $i \in [n-1]$. If $p_1 = p_n$, the sequence represents a copy of $C_{n-1}$. For two vertex sequences $x = (x_1,\ldots,x_i)$ and $y = (y_1,\ldots,y_j)$, we denote their \textit{concatenation} as $x \circ y := (x_1,\ldots,x_i,y_1,\ldots,y_j)$.

We call an edge-coloured path a \textit{good path} if it has at most one colour change. Similarly, we call an edge-coloured cycle a \textit{good cycle} if it has at most two colour changes. Note that a good path pn~$n$ vertices contains a good cycle on~$n$ vertices.

%%%%%%%%%%%%%%%%%%%%%%%%%%%%%%%%
% Preliminaries
%%%%%%%%%%%%%%%%%%%%%%%%%%%%%%%%
\section{Preliminaries}\label{section:prelim}

We will use the notion of a split colouring in bipartite graphs which was also used in \cite{gyarfas1983vertex, gyarfas1973ramsey, pokrovskiy2014partitioning}.

\begin{definition}[Split colouring]\label{def:split-colouring}
    Let $a,b \geq 2$. Let the edges of the complete bipartite graph $K_{a,b} = (V_1 \cup V_2, E)$ be coloured with two colours. We say that the colouring is \textup{split} if there exist partitions $V_1 = X \cup Y$ and $V_2 = Z \cup W$ such that all edges of $E(X,Z) \cup E(Y,W)$ have the same colour, and all edges of $E(X,W) \cup E(Y,Z)$ have the other colour.
\end{definition}

In \cite{pokrovskiy2014partitioning}, the following theorem was proven, which is an improvement on a theorem by Gy\'arf\'as and Lehel \cite{gyarfas1983vertex, gyarfas1973ramsey}.

\begin{theorem}[Theorem~1.8 in \cite{pokrovskiy2014partitioning}]\label{thm:path-partition-split}
    Let $k \geq 2$. Every edge-colouring of $K_{k,k} = (V_1 \cup V_2, E)$ with two colours satisfies one of the following statements:
    \begin{enumerate}[label=\abc]
        \item There are two monochromatic paths of different colours that cover all of $V_1 \cup V_2$.
        \item The colouring is split.
    \end{enumerate}
\end{theorem}

We will not use the theorem directly as stated, but the following simple corollary.

\begin{corollary}\label{cor:split.bipartit}
    Let $a,b \geq k \geq 2$. Every edge-colouring of $K_{a,b} = (V_1 \cup V_2, E)$ with two colours satisfies one of the following statements:
    \begin{enumerate}[label=\abc]
        \item\label{itm:split:a} There exists a good cycle on $2k$ vertices.
        \item The colouring is split.
    \end{enumerate}
\end{corollary}
\begin{proof}
    Write $G := K_{a,b}$. Assume~\ref{itm:split:a} does not hold. Then, for any subsets $A_1 \subset V_1$ and $A_2 \subset V_2$ of size $|A_1| = |A_2| = k$ we cannot find two vertex-disjoint paths on $A_1 \cup A_2$ that cover all of $A_1 \cup A_2$ and have different colours, because otherwise we could connect the endpoints of these paths and obtain a good cycle on $2k$ vertices. Thus, by Theorem~\ref{thm:path-partition-split}, the colouring on $G[A_1 \cup A_2]$ is split. Let the partitions $A_1 = X \cup Y$ and $A_2 = Z \cup W$ be given according to Definition~\ref{def:split-colouring}. Let $a \in A_1$ and $a' \in V_1 \setminus A_1$, and set $A_1' := (A_1 \cup \{a'\}) \setminus \{a\}$. Then, $G[A_1' \cup A_2]$ still must be split by the argument above. Since the partition $A_2 = Z \cup W$ is already determined by the other vertices in $A_1'$, this is only possible if all edges from $a'$ to $Z$ have the same colour, and all edges from $a'$ to $W$ have the other colour. But then $G[A_1 \cup \{a'\} \cup A_2]$ is split as well. As $a'$ was chosen arbitrarily, we can conclude that $G[V_1 \cup A_2]$ also has to be split, and by applying the same argument to $A_2$, the claim follows.
\end{proof}

We shall further make use of the following lemma, which is a consequence of Theoren~\ref{thm:path-cover}.

\begin{lemma}\label{lem:sum.r.g.b}
    Let $K_{a}$ be a complete graph on~$a$ vertices whose edges are coloured with red, green, and blue.
    There are values $\ell_c \in[a]$ and $p_c\in [2]$ for every $c \in \{red,green,blue\}$, such that the following statements hold.
    \begin{enumerate}[label=\rom]
    \item\label{itm:sum:i} For every $c \in \{red,green,blue\}$ there exist $p_c$ vertex-disjoint paths of colour $c$ in $K_a$ that cover $\ell_c$ vertices in total.
    \item\label{itm:sum:ii} $\ell_{red}+\ell_{green}+\ell_{blue}-p_{red}-p_{green}-p_{blue} \geq a-1$. 
    \end{enumerate}
\end{lemma}
\begin{proof}
    By applying Theorem~\ref{thm:path-cover} on $K_a$, we can find exactly three monochromatic disjoint paths $P_1, P_2, P_3$ in $K_a$ that cover all its vertices and are such that not all three have the same colour.
    Denote by $\mathcal{P}_c$ the subset of these paths that have colour $c$ for every $c \in \{red,green,blue\}$, and set $p_c = |\mathcal{P}_c|$ and $\ell_c = \sum_{P \in \mathcal{P}_c} |V(P)|$. Then~\ref{itm:sum:i} is true, and we have $p_c \leq 2$ for all $c \in \{red,green,blue\}$ as required. Now we have
    $$\Phi := \ell_{red} + \ell_{green} + \ell_{blue} - p_{red} - p_{green} - p_{blue} = a - 3,$$
    where some of the summands might be equal to $0$.
    We will demonstrate how to modify the families $\mathcal{P}_c$ in order to increase $\Phi$ by $2$ while still fulfilling~\ref{itm:sum:i} and satisfying $p_c \leq 2$ for all $c \in \{red,green,blue\}$.
    Let $e_1$ be the edge between an endpoint of $P_1$ and an endpoint of $P_2$, and let $e_2$ be the edge between the other endpoint of $P_2$ and an endpoint $P_3$ such that $P_1 \circ e_1 \circ P_2 \circ e_2 \circ P_3$ forms a Hamilton path. Let $c$ be the colour of $e_1$.
    
    \textbf{Case 1:} If neither $P_1$ nor $P_2$ have colour $c$, then we can add $e_1$ as a path to $\mathcal{P}_c$. In this way $\ell_c$ increases by $2$, while $p_c$ only increases by $1$, and hence $\Phi$ increases by $1$. If $P_2$ consists of a single vertex, we delete $P_2$ and let $e_1$ take the role of $P_2$ in the following. This does not change $\Phi$.
    
    \textbf{Case 2:} If exactly one of the paths has colour $c$, then we can extend this path by $e_1$. This way, only $\ell_c$ increases by $1$ while $p_c$ stays the same, and thus $\Phi$ again increases by $1$. 
    
    \textbf{Case 3:} Otherwise, both paths have colour $c$ and we can replace $P_1,P_2 \in \mathcal{P}_c$ by the concatenation $P_1 \circ P_2$. In this case, $p_c$ decreases by $1$ while $\ell_c$ stays the same, again resulting in an increase of $\Phi$ by $1$.

    By following the same argument for $e_2$, $P_2$ (which might have been updated), and $P_3$, we can increase $\Phi$ by $1$ again, and thus get $\Phi = a - 1$ as required. Note that $p_c \leq 2$ still holds for all $c \in \{red,green,blue\}$ even if both $e_1$ and $e_2$ are chosen to be included as new paths and they have the same colour, since this means that all paths $P_1$, $P_2$, $P_3$ have colours different from $e_1$ and $e_2$.

    Lastly, if $p_c = 0$ for some $c \in \{red,green,blue\}$, we can simply add a path consisting of a single vertex to $\mathcal{P}_c$, increasing $p_c$ and $\ell_c$ both by one, and thus not changing $\Phi$.
\end{proof}

We will also use the following result by Moon and Moser, which is a version of Dirac's Theorem for bipartite graphs that we will use in Section~\ref{section:upper.bound.cycle} to find long monochromatic cycles.
\begin{theorem}[Corollary 4 in \cite{moon1963hamiltonian}]\label{thm:moonmoser}
    Let $n \in \mathbb{N}$. Let $G = (A \cup B, E)$ be a bipartite graph with $|A| = |B| = n$, and $\deg(v) \geq \frac{n}{2}$ for each $v \in A \cup B$. Then, $G$ contains a Hamilton cycle.
\end{theorem}

%%%%%%%%%%%%%%%%%%%%%%%%%%%%%%%%
% Lower bound construction
%%%%%%%%%%%%%%%%%%%%%%%%%%%%%%%%
\section{The Lower Bound Construction}\label{section:lowerbound}
In this section, we provide an edge-colouring of $K_N$ with $N = \left \lfloor \frac{3n-2}{2} \right\rfloor - 1$ using exactly three colours, such that every copy of $P_n$ has at least two colour changes, yielding the required lower bounds for Theorem~\ref{thm:main_path} and Theorem~\ref{thm:main_even.cycle}.

\begin{proof}[Proof of the lower bounds for Theorem~\ref{thm:main_path} and Theorem~\ref{thm:main_even.cycle}]
Our construction is as follows (see also Figure~\ref{fig:construction}). Let \[N = \left \lfloor \frac{3n-2}{2} \right\rfloor - 1 = \left \lfloor \frac{3n-4}{2} \right\rfloor\,.\]
Take a partition $V(K_N)=X\cup Y\cup Z\cup W$ such that
$|X|=|Y|=|Z|= \left\lfloor \frac{n-2}{2} \right\rfloor$,
and $|W|=1$ if $n$ is even, or $|W|=2$ if $n$ is odd. Consider the colouring where the set $E(X \cup W) \cup E(Y,Z)$ is red, the set $E(Y) \cup E(Y,W) \cup E(X,Z)$ is green, and the set $E(Z) \cup E(Z,W) \cup E(X,Y)$ is blue. The \textit{template} of this construction is the complete graph $K_4$ on vertices $\{x,y,z,w\}$, with the edges $xw,yz$ coloured red, $xy,zw$ coloured blue, and $xz,yw$ coloured green.

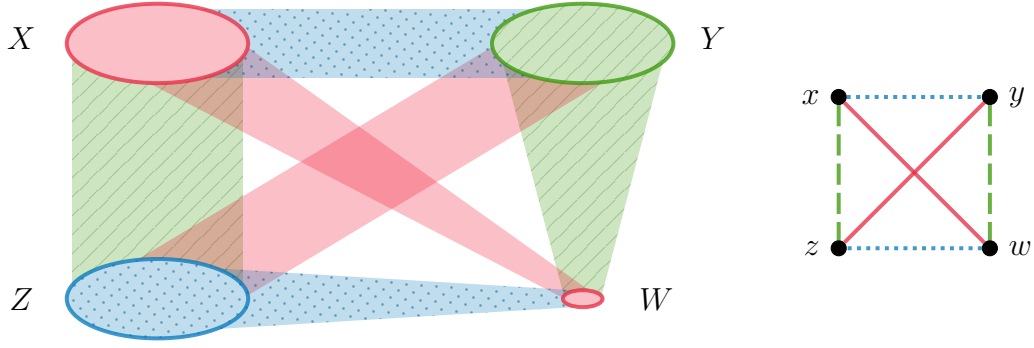
\begin{figure}
\begin{minipage}[c]{0.69\textwidth}
\centering
\begin{tikzpicture}[scale=1.5]

  \coordinate (X) at (0,    2.25);
  \coordinate (Y) at (3.75, 2.25);
  \coordinate (Z) at (0,    0);
  \coordinate (W) at (3.75, 0);

  \begin{scope}[even odd rule]
    \clip (-1, -0.55) rectangle (4.9, 2.8)
          (X) ellipse (0.8 and 0.35)
          (Y) ellipse (0.8 and 0.35)
          (Z) ellipse (0.8 and 0.35)
          (W) ellipse (0.175 and 0.075);

    % Green: solid fill + north east lines
    \fill[xgreen, opacity=0.25]
      (-0.75, 2.25) rectangle (0.75, 0);
    \fill[pattern={Lines[angle=45, distance=5pt, line width=0.3pt]}, pattern color=xgreen!60!black, opacity=0.5]
      (-0.75, 2.25) rectangle (0.75, 0);

    \fill[xgreen, opacity=0.25]
      (3.0, 2.25) -- (4.5, 2.25) -- (3.91, 0) -- (3.59, 0) -- cycle;
    \fill[pattern={Lines[angle=45, distance=5pt, line width=0.3pt]}, pattern color=xgreen!60!black, opacity=0.5]
      (3.0, 2.25) -- (4.5, 2.25) -- (3.91, 0) -- (3.59, 0) -- cycle;

    % Red: solid fill only
    \fill[xred, opacity=0.25]
      (-0.75, 2.25) -- (0.75, 2.25) -- (3.91, 0) -- (3.59, 0) -- cycle;
    \fill[xred, opacity=0.25]
      (3.0, 2.25) -- (4.5, 2.25) -- (0.75, 0) -- (-0.75, 0) -- cycle;

    % Blue: solid fill + crosshatch dots
    \fill[xblue, opacity=0.25]
      (0, 1.95) rectangle (3.75, 2.55);
    \fill[pattern={Dots[distance=4pt, radius=0.5pt, angle=40]},
          pattern color=xblue!70!black, opacity=0.7]
      (0, 1.95) rectangle (3.75, 2.55);

    \fill[xblue, opacity=0.25]
      (0, -0.30) -- (3.75, -0.07) -- (3.75, 0.07) -- (0, 0.30) -- cycle;
    \fill[pattern={Dots[distance=4pt, radius=0.5pt, angle=40]},
          pattern color=xblue!70!black, opacity=0.7]
      (0, -0.30) -- (3.75, -0.07) -- (3.75, 0.07) -- (0, 0.30) -- cycle;

  \end{scope}

  % --- X: red ellipse, solid only ---
  \fill[xred, opacity=0.25] (X) ellipse (0.8 and 0.35);
  \draw[xred!80!xgray, opacity=0.75, line width=1.5pt] (X) ellipse (0.8 and 0.35);

  % --- Y: green ellipse + north east lines ---
  \fill[xgreen, opacity=0.25] (Y) ellipse (0.8 and 0.35);
  \fill[pattern={Lines[angle=45, distance=5pt, line width=0.3pt]}, pattern color=xgreen!60!black, opacity=0.5]
    (Y) ellipse (0.8 and 0.35);
  \draw[xgreen, opacity=0.75, line width=1.5pt] (Y) ellipse (0.8 and 0.35);

  % --- Z: blue ellipse + patterns.meta dots ---
  \fill[xblue, opacity=0.25] (Z) ellipse (0.8 and 0.35);
  \fill[pattern={Dots[distance=4pt, radius=0.5pt, angle=40]},
        pattern color=xblue!70!black, opacity=0.7]
    (Z) ellipse (0.8 and 0.35);
  \draw[xblue, opacity=0.75, line width=1.5pt] (Z) ellipse (0.8 and 0.35);

  % --- W: red ellipse, solid only ---
  \fill[xred, opacity=0.25] (W) ellipse (0.175 and 0.075);
  \draw[xred!80!xgray, opacity=0.75, line width=1.5pt] (W) ellipse (0.175 and 0.075);

  % --- Labels ---
  \node[black] at (-1.2,  2.3) {$X$};
  \node[black] at ( 4.9, 2.3) {$Y$};
  \node[black] at (-1.2, 0)  {$Z$};
  \node[black] at ( 4.4, 0) {$W$};

\end{tikzpicture}
\end{minipage}
\hspace{-1.5cm}
\begin{minipage}[c]{0.3\textwidth}
\centering
\begin{tikzpicture}[vertex/.style={circle, draw, fill=black, inner sep=2pt},scale=2]
  \node[vertex,label=left:{$x$}] (x) at (0, 1)  {};
  \node[vertex,label=right:{$y$}] (y) at (1, 1)  {};
  \node[vertex,label=right:{$w$}] (w) at (1, 0)  {};
  \node[vertex,label=left:{$z$}] (z) at (0, 0)  {};
  \draw[xblue, dotted, opacity=0.75, line width=1.5pt] (x) -- (y);
  \draw[xred!80!xgray, opacity=0.75, line width=1.5pt] (x) -- (w);
  \draw[xgreen, dash pattern=on 8pt off 3pt, opacity=0.75, line width=1.5pt] (x) -- (z);
  \draw[xgreen, dash pattern=on 8pt off 3pt, opacity=0.75, line width=1.5pt] (y) -- (w);
  \draw[xred!80!xgray, opacity=0.75, line width=1.5pt] (y) -- (z);
  \draw[xblue, dotted, opacity=0.75, line width=1.5pt] (w) -- (z);
\end{tikzpicture}
\end{minipage}

\caption{The lower bound construction and its template: Dots represent blue, lines represent green, and red is without pattern.}
\label{fig:construction}
\end{figure}

This construction does not contain a good path on~$n$ vertices, and hence also not a good cycle on~$n$ vertices. Indeed, observe that $N - \left\lfloor \frac{n-2}{2} \right\rfloor = n - 1$. Hence, a path on~$n$ vertices in the construction has to intersect all three sets $X$, $Y$, $Z$.
If the path also intersects~$W$, then it cannot be good because the template of the construction does not contain a spanning good path. If, on the other hand, it does not intersect~$W$, then the only way to form a good path intersecting $X$, $Y$ and~$Z$ is to use some green $(Z,X)$ edges and some blue $(X,Y)$ edges, or one of the other two symmetric cases. But such a path can be of length at most $2|X|+1=2|Y|+1=2|Z|+1<n$.
\end{proof}

%%%%%%%%%%%%%%%%%%%%%%%%%%%%%%%%
% Proof Path
%%%%%%%%%%%%%%%%%%%%%%%%%%%%%%%%
\section{Proof of the Upper Bound for Paths}\label{section:upper.bound.path}

In this section, we will prove the upper bound for Theorem~\ref{thm:main_path}. We will first prove the following lemma, which allows us to either immediately find a good path on $n$ vertices as required, or leaves us with a special colouring described in Lemma~\ref{lemma:special.colouring}\ref{itm:special:b}. For an edge colouring of the complete graph~$K_N$, a \textit{colour class} is the set of edges~$E_c$ of one of the colours~$c$. A \textit{colour~$c$ component} is a connected component of the subgraph of~$K_N$ whose edge set is~$E_c$.

\begin{lemma}\label{lemma:special.colouring}
Let $n$ be sufficiently large, and let $k\in \left\{ \lceil \frac{n}{2} \rceil , \lceil \frac{n}{2} \rceil + 1 \right\}$ be even. If the edges of a complete graph on $N \geq 2k + \lceil \frac{n}{2} \rceil$ vertices are coloured with three colours, 
then at least one of the following properties holds.
\begin{enumerate}[label=\abc]
\item\label{itm:special:a} There is a good path on $n$ vertices.
\item\label{itm:special:b} There exists a partition of the vertex set $V(K_N)=X\cup Y\cup Z\cup W$
such that 
\begin{itemize}
\item $|X|+|Y|\geq \lceil \frac{n}{2} \rceil$ and $|Z|+|W|\geq \lceil \frac{n}{2} \rceil$,
\item the sets $E(X,Y)\cup E(Z,W)$ and $E(X,Z)\cup E(Y,W)$ and $E(X,W)\cup E(Y,Z)$
are subsets of pairwise different colour classes.
\end{itemize}
\end{enumerate}
\end{lemma}

\begin{proof}
Fix an arbitrary colouring of $K_N$ with three colours. As a first step, we want to show that~\ref{itm:special:a} holds, or we can find two vertex-disjoint monochromatic cycles, each of length $k$, that have the same colour. First, assume $k = \lceil \frac{n}{2} \rceil$. Since $k$ is even, we can apply Theorem~\ref{thm:ramsey_evencycle} for large enough $n$ to obtain a monochromatic cycle $C_1$ on $k$ vertices. Furthermore, since $N - k \geq 2k$, we can find another disjoint cycle $C_2$ on $k$ vertices, again applying Theorem~\ref{thm:ramsey_evencycle}. If $C_1$ and $C_2$ have different colours, say $C_1$ is red and $C_2$ is blue, then one of the following holds. Either there is a blue or red edge between $C_1$ and $C_2$, and we immediately obtain a good path on $n$ vertices. Otherwise, all edges between $C_1$ and $C_2$ have the third colour, forming a monochromatic copy of $K_{k,k}$ which includes a monochromatic path on $n$ vertices. This implies that if $k = \lceil \frac{n}{2} \rceil$, then~\ref{itm:special:a} holds, or $C_1$ and $C_2$ must have the same colour.

Now assume $k = \lceil \frac{n}{2} \rceil + 1$. We proceed similarly. At first we obtain a monochromatic cycle $C_1$ of length $k$ using Theorem~\ref{thm:ramsey_evencycle}. Afterwards, we are left with $N - k = 2k - 1$ vertices outside $V(C_1)$, one vertex short of applying Theorem~\ref{thm:ramsey_evencycle} to obtain a second disjoint monochromatic cycle $C_2$ of length $k$. To overcome this, we allow $|V(C_1) \cap V(C_2)| = 1$ by choosing an arbitrary vertex $w \in V(C_1)$ and using Theorem~\ref{thm:ramsey_evencycle} on $(V(K_n) \setminus V(C_1)) \cup \{w\}$. If both cycles cover $w$, we immediately get a good path on $2k - 1 > n$ vertices covering all of $V(C_1)$ and $V(C_2)$. Otherwise, $C_1$ and $C_2$ are disjoint and by the same argument as above we get that~\ref{itm:special:a} holds or both cycles have the same colour.
Thus, from now on we can assume that $C_1$ and $C_2$ are disjoint and have the same colour, say red. Then, let $C_1'$ and $C_2'$ be the red components containing $C_1$ and $C_2$, respectively. If $C_1' = C_2'$, we can find a monochromatic red path on at least $n$ vertices.

Otherwise, $C_1'$ and $C_2'$ are disjoint. Choose any partition $V(K_N) = V_1 \cup V_2$ such that $C_i' \subset V_i$ for $i \in [2]$, and every further red component of $V(K_N)$ is fully contained in either $V_1$ or $V_2$. Note that $|V_i| \geq |V(C_i)| \geq k$ for $i \in [2]$, and that there are no red edges between $V_1$ and $V_2$. Hence, by Corollary~\ref{cor:split.bipartit}, we know that either we obtain the desired good path for~\ref{itm:special:a} holds, or there exists a partition $V_1 = X \cup Y$ and $V_2 = Z \cup W$ such that all edges of $E(X,W) \cup E(Y,Z)$ have the same colour, say blue, and all edges of $E(X,Z) \cup E(Y,W)$ have the other colour, say green.

It remains to show that all edges in $E(X,Y) \cup E(W,Z)$ are red if~\ref{itm:special:a} does not hold. Assume otherwise, and w.l.o.g.~suppose that there is a blue edge between $X$ and $Y$. We show that we can find a good path on $n$ vertices, contradicting the assumption that~\ref{itm:special:a} does not hold. To do so, we choose $X' = \{x_1, x_2, \ldots, x_{|X'|}\} \subset X$, $Y' = \{y_1, y_2, \ldots, y_{|Y'|}\} \subset Y$, $Z' = \{z_1, z_2, \ldots, z_{|Z'|}\} \subset Z$, and $W' = \{w_1, w_2, \ldots, w_{|W'|}\} \subset W$ such that $|X'| + |Y'| = |Z'| + |W'| = \lceil \frac{n}{2} \rceil$, and the edge $y_1x_{|X'|}$ is blue. Note that $|X'| \leq |W'|$ or $|Y'| \leq |Z'|$, and w.l.o.g.~we can assume $|X'| \leq |W'|$. Then, 
\begin{align*}
 P_1 &:= (w_1, x_1, w_2, x_2, \ldots, w_{|X'|}, x_{|X'|}, y_1, z_1, y_2, z_2, \ldots, y_{|Z'|}, z_{|Z'|}, y_{|Z'| + 1}), \text{and} \\
 P_2 &:= (y_{|Z'| + 1}, w_{|X'| + 1}, y_{|Z'| + 2}, w_{|X'| + 2}, \ldots, y_{|Y'|}, w_{|W'|}) 
\end{align*} 
are both monochromatic paths sharing only the endpoint $y_{|Z'| + 1}$, and covering all of $X' \cup Y' \cup Z' \cup W'$. Thus, concatenating $P_1$ and $P_2$ leads to a good path on at least $n$ vertices, with the single colour change occurring at vertex $y_{|Z'| + 1}$.
\end{proof}

Note that the lower bound on $N$ in Lemma~\ref{lemma:special.colouring} is always satisfied if $N \geq 3\lceil\frac{n}{2}\rceil + 2$. If $4  \mid n$, $N \geq \frac{3n}{2}$ is already sufficient.

Next, we aim to show that a colouring as described in Lemma~\ref{lemma:special.colouring}\ref{itm:special:b} already contains a good path on $n$ vertices, and thus conclude Theorem~\ref{thm:main_path}. To do so, we first prove the following helpful lemma, which we will also use in Section~\ref{section:upper.bound.cycle} in the proof of Theorem~\ref{thm:main_even.cycle}.

\begin{lemma}\label{lemma:ABC.paths.cycles}
    Let $N \in \mathbb{N}$, and let the edges of $K_N$ be coloured with red, green, and blue. Let $A$, $B$, $C$ be disjoint subsets of $K_N$ such that $|A| + |C| > |B|$, and $E(A,B)$ and $E(B,C)$ are subsets of different colour classes.
    Let $f\in \{red,green,blue\}$ be the colour of $E(A,B)$,
    and let $g \in \{red,green,blue\}$ be the colour of $E(B,C)$.
    Moreover, assume that there exist
    $p^A_{f}$ vertex-disjoint paths of colour $f$ in $A$
    that together cover $\ell^A_{f}$ vertices, and
    that there exist
    $p^C_{g}$ vertex-disjoint paths of colour $g$ in $C$
    that together cover $\ell^C_{g}$ vertices. Then the following holds.
    \begin{enumerate}[label=\rom]
        \item\label{itm:ABC:i} If $|B|\geq p^A_{f} + p^C_{g} - 1$,
        then there is a good path on at least \[\min\{|A| + |B| + |C|, 2|B| + \ell^A_{f} + \ell^C_{g} + 1 - p^A_{f} - p^C_{g}\}\] vertices.
        \item\label{itm:ABC:ii} If $|B| \geq p^A_{f} + p^C_{g}$,
        then there is a good cycle on at least \[\min\{|A| + |B| + |C|, 2|B| + \ell^A_{f} + \ell^C_{g} - p^A_{f} - p^C_{g}\}\] vertices using only edges in $E(A) \cup E(A,B) \cup E(B,C) \cup E(C)$.
    \end{enumerate}
\end{lemma}

\begin{proof}
    Denote $A = \{a_1,\ldots,a_{|A|}\}$, $B = \{b_1,\ldots,b_{|B|}\}$, and $C = \{c_1,\ldots,c_{|C|}\}$. We consider two cases depending on the size of $B$. For the proof of~\ref{itm:ABC:i}, set $p := p^A_{f} + p^C_{g} - 1$, and for the proof of~\ref{itm:ABC:ii}, set $p := p^A_{f} + p^C_{g}$.
    
    \smallskip
    \textbf{Case 1:} Assume $|B| \geq (|A| - \ell^A_{f}) + (|C| - \ell^C_{g}) + p$. In this case we aim to find a good path or cycle, respectively, on $|A|+|B|+|C|$ vertices. Let $p_f^A \leq s \leq \ell^A_{f}$ and $p_g^C \leq t \leq \ell^C_{g}$ such that $|B| = (|A| - s) + (|C| - t) + p$. By assumption, we can find disjoint paths $P_1, \ldots, P_{p^A_{f}}$ in $A$ of colour $f$ that cover $s$ vertices in total, w.l.o.g., the vertices $a_i$ with $i > |A| - s$. Additionally, we can find disjoint paths $Q_1,\ldots,Q_{p^C_{g}}$ in $C$ of colour $g$ that cover $t$ vertices, w.l.o.g., the vertices $c_i$ with $i > |C| - t$. Relabel $B=\{b_1,\ldots,b_{p},b'_1,\ldots,b'_{|A|-s},b''_1,\ldots,b''_{|C|-t}\}$.
    For the proof of~\ref{itm:ABC:i}, we find the path
    \begin{align*}
        P_1 &\circ (b_1) \circ \ldots \circ P_{p^A_f} \circ (b_{p^A_f}) \circ (a_1,b_1',\ldots,a_{|A| - s},b_{|A| - s}')\\
        &\circ (c_1, b_1'',\ldots,c_{|C| - t},b_{|C|-t}'') \circ Q_1 \circ (b_{p^A_f + 1}) \circ \ldots \circ (b_{p^A_f + p^C_g - 1}) \circ Q_{p^C_g},
    \end{align*}
    which has a single colour change at $b_{|A|-s}'$, and covers $|A| +  |B| + |C|$ vertices.
    For the proof of~\ref{itm:ABC:ii}, we find the cycle
    \begin{align*}
        (b_p) &\circ P_1 \circ (b_1) \circ \ldots \circ P_{p^A_f} \circ (b_{p^A_f}) \circ (a_1,b_1',\ldots,a_{|A| - s},b_{|A| - s}')\circ (c_1, b_1'',\ldots,c_{|C| - t},b_{|C|-t}'')\\
        &\circ Q_1 \circ (b_{p^A_f + 1}) \circ \ldots \circ (b_{p^A_f + p^C_g - 1}) \circ Q_{p^C_g} \circ (b_p), 
    \end{align*}
    which has colour changes at $b_{|A| - s}'$ and $b_p$ and covers $|A| + |B| + |C|$ vertices as well.

    \smallskip
    \textbf{Case 2:} Assume $|B| < (|A| - \ell^A_{f}) + (|C| - \ell^C_{g}) + p$. In this case, we aim to find a good path on $2|B| + \ell^A_{f} + \ell^C_{g} + 1 - p^A_{f} - p^C_{g}$ vertices or a good cycle on $2|B| + \ell^A_{f} + \ell^C_{g} - p^A_{f} - p^C_{g}$ vertices, respectively.    
    Let $0 \leq s \leq |A| - \ell_f^A$ and $0 \leq t \leq |C| - \ell_g^C$ such that $|B| = s + t + p.$ By assumption, we can find disjoint paths $P_1,\ldots,P_{p^A_f}$ in $A$ of colour $f$ that cover $\ell_f^A$ vertices, w.l.o.g., the vertices $a_i$ with $i > |A| - \ell_f^A$. In the same way, we find disjoint paths $Q_1,\ldots,Q_{p^C_g}$ in $C$ of colour $g$ that cover $\ell_g^C$ vertices, w.l.o.g., the vertices $c_i$ with $i > |C| - \ell_g^C$. Relabel $B = \{b_1,\ldots,b_p,b_1',\ldots,b_s',b_1'',\ldots,b_t''\}$. For the proof of~\ref{itm:ABC:i} we find the path
    \begin{align*}
        P_1 &\circ (b_1) \circ \ldots \circ P_{p^A_f} \circ (b_{p^A_f}) \circ (a_1,b_1',\ldots,a_s,b_s')\\
        &\circ (c_1,b_1'',\ldots,c_t,b_t'') \circ Q_1 \circ (b_{p^A_f + 1}) \circ \ldots \circ (b_{p^A_f + p^C_g - 1}) \circ Q_{p^C_g}
    \end{align*}
    which has only a colour change at $b_s'$. Its number of vertices is
    $$(\ell_f^A + s) + |B| + (\ell_g^C + t) = \ell_f^A + 2|B| + \ell_g^C - p_f^A - p_g^C + 1$$ as required.
    For the proof of~\ref{itm:ABC:ii}, we find the cycle 
    \begin{align*}
        (b_p) &\circ P_1 \circ (b_1) \circ \ldots \circ P_{p^A_f} \circ (b_{p^A_f}) \circ (a_1,b_1',\ldots,a_s,b_s')\\
        &\circ (c_1,b_1'',\ldots,c_t,b_t'') \circ Q_1 \circ (b_{p^A_f + 1}) \circ \ldots \circ (b_{p^A_f + p^C_g - 1}) \circ Q_{p^C_g} \circ (b_p)
    \end{align*}
    which has colour changes at $b_s'$ and $b_p$, and covers $$(\ell_f^A + s) + |B| + (\ell_g^C + t) = \ell_f^A + 2|B| + \ell_g^C - p_f^A - p_g^C$$ vertices.
\end{proof}

With Lemma~\ref{lemma:ABC.paths.cycles} at hand, we are ready to prove the following theorem.

\begin{theorem}\label{thm:good.path.split.colouring}
For $n$ sufficiently large, assume the complete graph on 
$N \geq \lfloor \frac{3n-2}{2} \rfloor$
vertices
is coloured with red, blue and green
such that there exists a partition
$V(K_N)=X\cup Y\cup Z\cup W$
with
$|X|+|Y|\geq \lceil \frac{n}{2} \rceil$ 
and $|Z|+|W|\geq \lceil \frac{n}{2} \rceil$
such that
the set $E(X,Y)\cup E(Z,W)$ is red, 
the set $E(X,Z)\cup E(Y,W)$ is green, 
and the set $E(X,W)\cup E(Y,Z)$ is blue.
Then there is
a good path on $n$ vertices.
\end{theorem}
\begin{proof}
    We can assume that $N = \lceil \frac{3n - 2}{2} \rceil$ as otherwise we can delete vertices of $K_N$ until we have the desired number of vertices while maintaining $|X| + |Y| \geq \lceil \frac{n}{2} \rceil$ and $|Z| + |W| \geq \lceil \frac{n}{2} \rceil$. Denote $X = \{x_1,x_2,\ldots,x_{|X|}\}$, $Y = \{y_1,y_2,\ldots,y_{|Y|}\}$, $Z = \{z_1,z_2,\ldots,z_{|Z|}\}$, and $W = \{w_1,w_2,\ldots,w_{|W|}\}$. Set $\mathcal{S} := \{X,Y,Z,W\}$. We first show that we can find a good path on $n$ vertices if $|A| \geq \lfloor \frac{n}{2} \rfloor$ for any $A \in \mathcal{S}$. W.l.o.g., assume $|X| \geq \lfloor \frac{n}{2} \rfloor$. By assumption, $|Z| + |W| \geq \lceil \frac{n}{2} \rceil$. If $|Z| \geq \lceil \frac{n}{2} \rceil$ or $|W| \geq \lceil \frac{n}{2} \rceil$, we immediately get a monochromatic path on at least $n$ vertices in $E(X,Z)$ or $E(X,W)$, respectively. Otherwise, we get the path
    $$(z_1,x_1,z_2,x_2,\ldots,z_{|Z|},x_{|Z|},w_1,x_{|Z|+1},w_2,x_{|Z|+2},\ldots,w_{\lfloor \frac{n}{2}\rfloor - |Z|},x_{\lfloor \frac{n}{2}\rfloor})$$ for even $n$, and 
    $$(z_1,x_1,z_2,x_2,\ldots,z_{|Z|},x_{|Z|},w_1,x_{|Z|+1},w_2,x_{|Z|+2},\ldots,w_{\lfloor \frac{n}{2}\rfloor - |Z|},x_{\lfloor \frac{n}{2}\rfloor},w_{\lceil \frac{n}{2}\rceil - |Z|})$$ for odd $n$, each on $n$ vertices and with a single colour change at $x_{|Z|}$.

    Thus, from now on we can assume $|A| \leq \lfloor \frac{n - 2}{2} \rfloor$ for all $A \in \mathcal{S}$. As a next step, we want to show that $|A| \geq 3$ for all $A \in \mathcal{S}$. Assume otherwise, and suppose, w.l.o.g., $|W| \leq 2$. By $\sum_{A \in \mathcal{S}} |A| = N$, we can deduce that this is actually only possible if $n$ is even and $|W| = 2$, $|X| = |Y| = |Z| = \frac{n-2}{2}$. Due to symmetry, we may assume that the single edge in $W$ is blue. Then $(w_1,w_2,x_1,y_1,x_2,y_2,\ldots,x_{\frac{n-2}{2}},y_{\frac{n-2}{2}})$ forms a good path with the single colour change occurring at $x_1$. Hence, we can assume $|A| \geq 3$ for all $A \in \mathcal{S}$ for the remainder of the proof. Furthermore, for distinct sets $A,B,C,D \in \mathcal{S}$, we have $|A| + |C| = N - (|B| + |D|) \geq \lfloor \frac{3n-2}{2} \rfloor - 2\lfloor \frac{n-2}{2} \rfloor \geq \frac{n+2}{2}$. In particular, for any distinct sets $A,B,C \in \mathcal{S}$, we have $|A| + |C| > |B|$, so we can apply Lemma~\ref{lemma:ABC.paths.cycles} for any such triple of sets.
    
    We want to use Lemma~\ref{lemma:ABC.paths.cycles}\ref{itm:ABC:i} to find good paths of length at least $n$. 
    For this, we need that each $A \in \mathcal{S}$ contains suitable monochromatic paths covering enough vertices of $A$.
    These paths are provided by Lemma~\ref{lem:sum.r.g.b}.
    More precisely, for every $A \in \mathcal{S}$ and every colour $c \in \{red,green,blue\}$, fix values $\ell^A_c \in \big[|A|\big]$ and $p^A_c \in [2]$ as guaranteed by Lemma~\ref{lem:sum.r.g.b}. 
    By the conclusion of Lemma~\ref{lem:sum.r.g.b}, we can apply Lemma~\ref{lemma:ABC.paths.cycles}
    for every choice of distinct $A,B,C \in \mathcal{S}$ to either find a good path on at least $|A| + |B| + |C|$ vertices, or a good path, let's call it $P_{B-\{A,C\}}$, on at least $\ell^A_f + 2|B| + \ell^C_g + 1 - p^A_f - p^C_g$ vertices, where $f \in \{red,green,blue\}$ is the colour of $E(A,B)$ and $g \in \{red,green,blue\}$ is the colour of $E(B,C)$. Note that the size condition on $B$ in Lemma~\ref{lemma:ABC.paths.cycles}\ref{itm:ABC:i} is met because $|B| \geq 3$, and $p^A_f + p^C_g \leq 4$.

In case the obtained good path for a given triple $A,B,C \in \mathcal{S}$ has at least $|A| + |B| + |C|$ vertices, the path actually covers at least $n$ vertices, since $|A| + |B| + |C| = N - |D| \geq \lfloor \frac{3n-2}{2} \rfloor - \lfloor \frac{n-2}{2} \rfloor = n$ for $D \in \mathcal{S} \setminus \{A,B,C\}.$ So assume the other case occurs for all triples $A,B,C \in \mathcal{S}$.

We want to show that at least one of these paths covers $n$ vertices. Summing up the sizes of all
$12$ paths, we obtain
\begin{align*}
\sum_{\substack{B \in \mathcal{S} \\ \{A,C\} \subset \mathcal{S} \setminus \{B\}}} |V(P_{B-\{A,C\}})| & \geq 2 \sum_{A\in \mathcal{S}} (\ell^A_{red} + \ell^A_{green} + \ell^A_{blue} - p^A_{red} - p^A_{green} - p^A_{blue}) + 3\sum_{B\in \mathcal{S}} 2|B| + 12 \\
& \geq 2 \sum_{A\in \mathcal{S}} (|A|-1) + 6 \cdot \sum_{B\in \mathcal{S}} |B| + 12& \\
& = 8N + 4 \geq  12n - 8\,,
\end{align*}
where we use Lemma~\ref{lem:sum.r.g.b}\ref{itm:sum:ii} in the second inequality.

Hence, among the $12$ paths considered in the sum, at least one path must have at least $n$ vertices, finishing the proof.
\end{proof}

Theorem~\ref{thm:main_path} follows by an application of Lemma~\ref{lemma:special.colouring}, which leads to a good path on $n$ vertices or a special colouring that contains a good path on $n$ vertices by Theorem~\ref{thm:good.path.split.colouring}.
%%%%%%%%%%%%%%%%%%%%%%%%%%%%%%%%
% Proof Cycle
%%%%%%%%%%%%%%%%%%%%%%%%%%%%%%%%
\section{Proof of the Upper Bound for Cycles}\label{section:upper.bound.cycle}
In this section, we will prove the upper bound for Theorem~\ref{thm:main_even.cycle}. The approach is very similar to the proof for paths. As before, we want to find a good cycle in two steps. First, we show that for any colouring we either find a good cycle, or the colouring has a very specific structure. Second, we show that given a colouring of this specific structure we find a good cycle as well.

To prepare for the first step, we use the following helpful lemma about paths and cycles in bipartite graphs.
    
\begin{lemma}\label{lemma:Dirac.variant}
Let $t,N\in \mathbb{N}$ with $t\leq N/8$.
Let $G=(A\cup B,E)$ be a bipartite graph with $|A|=|B|=N$
and $\delta(G)\geq \frac{N}{2} - t$. Then at least one of the following statements holds.
\begin{enumerate}[label=\rom]
\item\label{itm:Dirac:i} $G$ contains a path on at least $2N-8t$ vertices.
\item\label{itm:Dirac:ii} $G$ contains two vertex-disjoint cycles, each of length at least $N-2t$.
\end{enumerate}
\end{lemma}
\begin{proof}
    Let $A = \{a_1,\ldots,a_N\}$ and $B = \{b_1,\ldots,b_N\}$. Let $P$ be a longest path in $G$. We have $|V(P)| \geq 2\delta(G) \geq N - 2t$, and also we can assume $|V(P)| < 2N - 8t$ as otherwise~\ref{itm:Dirac:i} would hold. We make a case distinction depending on whether the endpoints of $P$ lie in the same or in different partition classes of $G$.

    \textbf{Case 1:} Assume first that $P$ has exactly one endpoint in each of the sets $A$ and $B$. Then, w.l.o.g., $P = (a_1,b_1,\ldots,a_m,b_m)$ with $m = |V(P)|/2 \in [N/2 - t, N - 4t - 1]$. By maximality of $P$, all neighbours of $a_1$ must be in $\{b_1,\ldots,b_m\}$, and all neighbours of $b_m$ must be in $\{a_1,\ldots,a_m\}$. Since $\delta(G) \geq \frac{N}{2} - t > \frac{m}{2}$, there must be an index $i \in [m]$ such that $a_1b_i,a_ib_m \in E(G)$. Then
    $$C_1 = (a_1,b_i,a_{i+1},b_{i+1},\ldots,a_m,b_m,a_i,b_{i-1},a_{i-1},\ldots,b_1,a_1)$$ forms a cycle of length at least $N - 2t$. Now, if there was an edge between $V(P)$ and $V(G) \setminus V(P)$, we could find a path $P'$ with $|V(P')| > |V(P)|$ contradicting the maximality of $P$. 
    Hence, we may assume from now on
    that there is no edge between
    $V(P)$ and $V(G) \setminus V(P)$.
    Consider the bipartite graph $G' = (A' \cup B', E')$ induced by $V(G) \setminus V(P)$. We have $|A'|,|B'| \geq \delta(G) \geq \frac{N}{2} - t$. On the other hand, we have $|A'|,|B'| = N - m \leq N - (N/2 -t) = N/2 + t$. Thus, $\delta(G) \geq |A'| / 2$ and by Theorem~\ref{thm:moonmoser}, $G'$ has a Hamilton cycle $C_2$ of length at least $N - 2t$, and hence~\ref{itm:Dirac:ii} holds.

    \textbf{Case 2:} Assume now that $P$ has both endpoints in the same set, say $A$. W.l.o.g., we can assume $P = (a_1,b_1,\ldots,b_{m-1},a_m)$ with $m = (|V(P)| + 1)/2 \in [N/2 - t + 1, N - 4t]$. Again, by maximality of $P$, all neighbours of $a_1$ and $a_m$ have to be in $\{b_1,\ldots,b_{m-1}\}$. Since $\delta(G) \geq \frac{N}{2} - t > \frac{m}{2}$, we find $i \in [m]$ such that $a_1b_{i+1}, a_mb_i \in E(G)$. Then
    $$C_1 = (a_1,b_1,\ldots,a_i,b_i,a_m,b_{m-1},a_{m-1},\ldots,a_{i+2},b_{i+1},a_1)$$
    is a cycle with vertex set $V(P) \setminus \{a_{i+1}\}$ having length at least $2m-2 \geq N - 2t$. We want to show that we can extend $C_1$ to a path of length $2m > |V(P)|$ unless the neighbourhoods of the vertices in $V(G) \setminus V(C)$ behave in a specific way. For simplicity, relabel $C_1 = (c_1,c_2,\ldots,c_{2m-2},c_1)$ with $c_{2k - 1} \in A$ and $c_{2k} \in B$ for $k \in [m-1]$. If there was a vertex $v \in V(G) \setminus V(C_1)$ with a neighbour $c_j \in V(C_1)$ and a neighbour $w \in V(G) \setminus V(C_1)$, then $$(w,v,c_j,c_{j+1},\ldots,c_{2m-2},c_1,c_2,\ldots,c_{j-1})$$ would be a path of length $2m$, which contradicts the maximality of $P$. So, from now on assume that every vertex in $V(G) \setminus V(C_1)$ has all its neighbours either in $V(G) \setminus V(C_1)$, or in $V(C_1)$. It follows that $a := a_{i+1}$ has at least $\delta(G)$ neighbours in $V(C_1) \cap B$. If there was another vertex $b \in B' := B \setminus V(C_1)$ with all its at least $\delta(G)$ neighbours in $V(C_1) \cap A$, we would find $j \in [m-1]$ such that $ac_{2j},bc_{2j-1} \in E(G)$, and thus we would find a path $$(a,c_{2j},c_{2j+1},\ldots,c_{2m-2},c_1,c_2,\ldots,c_{2j-1},b)$$ of length $2m$, again contradicting the maximality of $P$. It follows that $N(B') \subset A \setminus V(C_1)$, and $N(N(B')) \subset B'$ because otherwise a vertex in $N(B')$ would have neighbours in $V(C_1)$ and $V(G) \setminus V(C_1)$. Hence, 
    applying Theorem~\ref{thm:moonmoser} analogously to the first case, we can conclude that $G[N(B') \cup B']$ contains a Hamilton cycle $C_2$ with length at least $N - 2t$.
\end{proof}

Next, we prove a lemma very similar to Lemma~\ref{lemma:special.colouring}, but instead of a good path or a special colouring, we aim for a good even cycle or a special colouring.
We say that two edges are \textit{independent} if they do not share a vertex.

\begin{lemma}\label{lemma:special.colouring.cycle}
Let $\varepsilon > 0$ and~$n$ be sufficiently large and even. If the edges of a complete graph on 
$N \geq (\frac32+\varepsilon)n$ 
vertices are coloured with three colours, say red, blue, and green, 
then at least one of the following properties hold.
\begin{enumerate}[label=\abc]
\item There is a good cycle on exactly $n$ vertices.
\item There exists a partition $X\cup Y\cup Z\cup W \subset V(K_N)$ of a subset of $V(K_N)$ such that 
\begin{itemize}
\item $|X|+|Y|\geq (\frac12 + 0.1\varepsilon)n$ and $|Z|+|W|\geq (\frac12 + 0.1\varepsilon)n$,
\item the sets $E(X,Y)\cup E(Z,W)$ and $E(X,Z)\cup E(Y,W)$ and $E(X,W)\cup E(Y,Z)$
are subsets of different colour classes.
\end{itemize}
\end{enumerate}
\end{lemma}
\begin{proof}
    Fix an arbitrary colouring of $K_N$. We proceed similarly to the proof of Lemma~\ref{lemma:special.colouring}, first looking for two disjoint monochromatic cycles $C_1$ and $C_2$ of a certain length. Set $\delta := \varepsilon / 4$ and choose $k \in \big[(\frac{1}{2} + \delta)n,(\frac{1}{2} + \delta)n + 2\big]$ to be even. For sufficiently large $n$, we can apply Theorem~\ref{thm:ramsey_evencycle} to find a monochromatic cycle~$C_1$ of length $k$. Moreover, since $N - k > 2k$, we can find another disjoint monochromatic cycle~$C_2$ of length $k$ on the remaining vertices. We distinguish several cases.

    \textbf{Case 1:} Assume $C_1$ and $C_2$ have the same colour, say red. If there is a red edge between $V(C_1)$ and $V(C_2)$, we immediately obtain a monochromatic red path on $n$ vertices. Adding the edge between the endpoints can introduce at most two colour changes, giving the desired good cycle on $n$ vertices. Otherwise, all edges between $V(C_1)$ and $V(C_2)$ are green or blue. Then, by Corollary~\ref{cor:split.bipartit}, we either find a good path on $2k$ vertices, or the green-blue colouring between $V(C_1)$ and $V(C_2)$ is split. If we find a good path $P$ on $2k > n$ vertices, we can shorten it to exactly $n$ vertices. As $n$ is even, one of $P$'s endpoints lies in $V(C_1)$ and the other lies in $V(C_2)$. Therefore, the edge between the endpoints is blue or green, and adding this edge produces a good cycle on $n$ vertices, as required. 

    So from now on, assume that the colouring between $V(C_1)$ and $V(C_2)$ is split, i.e.~there is a partition $V(C_1) = X \cup Y$ and $V(C_2) = Z \cup W$ such that $E(X,W) \cup E(Y,Z)$ is blue and $E(X,Z) \cup E(Y,W)$ is green. We consider two sub-cases.
    
    \textbf{Case 1.1:} Assume that we can find two non-red independent edges of the same colour in $E(X,Y)$ or in $E(Z,W)$. W.l.o.g., assume there are two independent blue edges in $E(X,Y)$. Then $|X|,|Y| \geq 2$. Moreover, w.l.o.g.~we may also assume that $|X| \geq |W|$ and hence $|Y|\le|Z|$. Otherwise, we could exchange roles of $X$ and $Y$, and of $Z$ and $W$. Then, we construct a good cycle on $n$ vertices as follows:
    Choose integers $s_X \leq |X|, s_Y \leq |Y|, s_Z \leq |Z|, s_W \leq |W|$ with $s_X + s_Y = \frac{n}{2} + 1$, $s_Z + s_W = \frac{n}{2} - 1$, and $s_X,s_Y \geq 2$,
    and $s_X>s_W$. Label the vertices $X =\{x_1,x_2,\ldots,x_{|X|}\}$, $Y =\{y_1,y_2,\ldots,y_{|Y|}\}$, $Z =\{z_1,z_2,\ldots,z_{|Z|}\}$, and $W =\{w_1,w_2,\ldots,w_{|W|}\}$ such that $x_{s_W + 1}y_1$ and $x_1y_{s_Y}$ are blue. Then the cycle
    \begin{align*}
        (&x_1,w_1,x_2,w_2,\ldots,w_{s_W},x_{s_W + 1},y_1,z_1,y_2,z_2,\ldots,y_{s_Y - 1},\\
        &z_{s_Y - 1},x_{s_W + 2},z_{s_Y},x_{s_W + 3},z_{s_Y + 1}, \ldots, x_{s_X},z_{s_Z},y_{s_Y},x_1)
    \end{align*}
    is a good cycle on $n$ vertices with colour changes at $z_{s_Y - 1}$ and $z_{s_Z}$. (Note that if $s_X = s_W + 1$ we go from $z_{s_Y-1} = z_{s_Z}$ to $y_{s_Y}$. If $s_X = s_W + 2$, we go from $x_{s_W + 2}$ to $z_{s_Z}$.)

    \textbf{Case 1.2:} Assume we cannot find two independent non-red edges of the same colour in $E(X,Y)$ or in $E(Z,W)$. Then, by deleting at most two vertices from each of the sets $X,Y,Z$, and $W$, we eliminate all green and blue edges in $E(X,Y) \cup E(Z,W)$. This yields the desired partition.

    \textbf{Case 2:} Assume $C_1$,$C_2$ have different colours, say $C_1$ is blue, and $C_2$ is green. In the proof of Lemma~\ref{lemma:special.colouring}, the presence of even a single green or blue edge between the two cycles sufficed to produce a good path. Otherwise, the bipartite graph between them was completely red, which trivially contained a monochromatic path. For cycles, however, the situation is not so simple: even multiple green and blue edges do not necessarily yield a good cycle on exactly $n$ vertices. We therefore distinguish whether there are many red edges between $C_1$ and $C_2$ or not.

    \textbf{Case 2.1:} Assume more than $2\delta^{-1}$ vertices in $V(C_1)$ have less than $\frac12 k$ incident red edges with the other endpoint in $V(C_2)$. In other words, each of these vertices has more than $\frac12 k$ edges which are green or blue towards $C_2$. Among them, choose two vertices $v_1$ and $v_2$ whose distance along $C_1$ is at most $\delta n$, let $t := \mathrm{dist}_{C_1}(v_1,v_2)$. Further, set $m := 2k - n \geq 2\delta n$. If we can find vertices $w_1,w_2 \in V(C_2)$ with $\mathrm{dist}_{C_2}(w_1,w_2) = m - t + 2$, such that both $v_1w_1$ and $v_2w_2$ are not red, we find a good cycle on $n$ vertices as follows: Take $C_1 \cup C_2$, and add the edges $v_1w_1$, $v_2w_2$. Then, delete the $t-1$ vertices between $v_1$ and $v_2$ on $C_1$, and the $m - t + 1$ vertices between $w_1$ and $w_2$ on $C_2$. The resulting cycle has $2k - (t-1) - (m - t + 1) = 2k - m = n$ vertices. For each $i \in [2]$, a colour change occurs at either $v_i$ or $w_i$, giving exactly two colour changes, as required.

    It remains to show that such vertices $w_1, w_2 \in V(C_2)$ exist. For this, define \[W_i = \big\{ w \in V(C_2) : v_iw \text{ is blue or green}\big\},\] and \[W_i' = \big\{u \in V(C_2): \exists w \in W_i \text{ with } \mathrm{dist}_{C_2}(w,u) = m - t + 2\big\}\] for $i \in [2]$. Note that $|W_i'| \geq |W_i|$ for each $i \in [2]$, since every vertex in $W_i$ produces a unique vertex of $W_i'$ by moving $m - t + 2$ steps in a fixed direction along $C_2$. By assumption, $|W_i| > \frac{k}{2}$ for $i \in [2]$, and therefore $|W_i'| > \frac{k}{2}$ as well. Hence, we find $w_1 \in W_1 \cap W_2'$. Since $w_1 \in W_2'$, we find $w_2 \in W_2$ such that $\mathrm{dist}_{C_2}(w_1,w_2) = m - t + 2$. Moreover, none of the edges
    $v_1w_1$ and $v_2w_2$ is red, as
    $w_i\in W_i$ for both $i\in [2]$,
    and hence the requirements described above are fulfilled. Therefore, we obtain a good cycle on $n$ vertices.

    \textbf{Case 2.2:} Assume at most $2\delta^{-1}$ vertices in $V(C_1)$ have less than $\frac12 k$ incident red edges with the other endpoint in $V(C_2)$. For $i \in [2]$, we create a subset $V_i \subset V(C_i)$ by deleting $\lceil 2\delta^{-1} \rceil$ vertices of $C_i$ with the smallest number of incident red edges to $C_{3-i}$. Let $G_R$ be the bipartite graph with vertex classes $V_1$ and $V_2$ and edge set consisting of all red edges of $E(V_1,V_2).$ Then, $$|V_1| = |V_2| = N' := k - \lceil 2 \delta^{-1} \rceil \text{ and } \delta(G_R) \geq \frac12 k - \lceil 2 \delta^{-1} \rceil \geq \frac{N'}{2} - \lceil 2 \delta^{-1} \rceil.$$

    Let $P$ be a longest path in $G_R$. If $|V(P)| \geq 2(k - \delta n) \geq n$, we can shorten $P$ to have exactly $n$ vertices, and add the edge between the endpoints. Since $P$ is completely red, this leads to a cycle on $n$ vertices with at most two colour changes. Otherwise, we can assume that $|V(P)| < 2(k - \delta n) \leq 2N' - 8 \lceil 2 \delta^{-1} \rceil.$ By Lemma~\ref{lemma:Dirac.variant} with $t = \lceil 2 \delta^{-1} \rceil$, the graph $G_R$ contains two disjoint cycles $C_1'$ and $C_2'$ each of length at least 
    $k' := N' - 2t = k - O(\delta^{-1}) > (\frac12 + \frac{\delta}{2})n$. Now, we can repeat the argument of Case~1 with $C_1',C_2',k'$ in place of $C_1,C_2,k$.
\end{proof}

\begin{remark}
    Note that in the above proof we only use the assumption of $n$ being even in Case~1,
    and this happens in two places only: (1)
    in the first paragraph of Case 1, when we conclude that the path $P$ has one endpoint in each of the sets $V(C_1)$ and
    $V(C_2)$, and (2) when we construct the good cycle in Case 1.1. If $n$ is odd,
    we would still be able to construct a good cycle in Case 1.1 as follows: choose $s_X + s_Y = \lceil \frac{n}{2} \rceil$ and $s_W + s_Z = \lfloor \frac{n}{2} \rfloor$, and create a cycle as before, but return directly from $z_{s_Z}$ to $x_1$ instead of passing through $y_{s_Y}$.
    Thus, it would be enough to fix the argument in (1) for odd $n$ in order to prove Lemma~\ref{lemma:special.colouring.cycle} independent of the parity of $n$.
\end{remark}

For the proof of Theorem~\ref{thm:main_even.cycle} we now only need the following theorem, which states that given the special colouring from Lemma~\ref{lemma:special.colouring.cycle} we find a good cycle on $n$ vertices. We prove the theorem for odd and even cycles so that one would obtain Theorem~\ref{thm:main_even.cycle} for all cycles, if one could prove Lemma~\ref{lemma:special.colouring.cycle} for odd cycles. 

\begin{theorem}\label{thm:cycle.split.colouring}
    Let $n$ be sufficiently large. Let a complete graph on 
    $N \geq \frac{3}{2}n + 4$
    vertices
    be coloured with red, blue and green
    such that there exists a partition
    $V(K_N)=X\cup Y\cup Z\cup W$
    with
    $|X|+|Y|\geq \frac{n}{2}$ 
    and $|Z|+|W| \geq \frac{n}{2}$
    such that
    the set $E(X,Y)\cup E(Z,W)$ is red, 
    the set $E(X,Z)\cup E(Y,W)$ is green, 
    and the set $E(X,W)\cup E(Y,Z)$ is blue.
    Then there is
    a good cycle on exactly $n$ vertices.
\end{theorem}

\begin{proof}
    Denote $X = \{x_1,\ldots,x_{|X|}\}$, $Y = \{y_1,\ldots,y_{|Y|}\}$, $Z = \{z_1,\ldots,z_{|Z|}\}$, and $W = \{w_1,\ldots,w_{|W|}\}$, and set $\mathcal{S} = \{X,Y,Z,W\}$. We distinguish two cases.

    \textbf{Case 1:} Assume $|A| \geq \frac{n}{2}$ for any $A \in \mathcal{S}$. W.l.o.g., suppose $|X| \geq \frac{n}{2}$. 
    
    First, consider the sub-case that $n$ is even. Since $|Z| + |W| \geq \frac{n}{2}$ by assumption, we can choose integers $0 \leq s_Z \leq |Z|$ and $0 \leq s_W \leq |W|$ with $s_Z + s_W = \frac{n}{2}$. Then,
    $$(x_1,z_1,x_2,z_2,\ldots,x_{s_Z},z_{s_Z},x_{s_Z+1},w_1,x_{s_Z + 2},w_2,\ldots,x_{\frac{n}{2}},w_{s_W},x_1)$$
    forms a cycle on $n$ vertices with colour changes at $x_1$ and $x_{s_Z + 1}$. 
    
    Now, consider the sub-case that~$n$ is odd. Note that $|X| \geq \lceil \frac{n}{2} \rceil$ in that case. If there is an edge, say $x_1x_{\lceil \frac{n}{2} \rceil} \in E(X)$ which is blue or green, we can choose values $0 \leq s_Z \leq |Z|$ and $0 \leq s_W \leq |W|$ with $s_Z + s_W = \lfloor\frac{n}{2}\rfloor$, and find the cycle
    $$(x_1,z_1,x_2,z_2,\ldots,x_{s_Z},z_{s_Z},x_{s_Z+1},w_1,x_{s_Z + 2},w_2,\ldots,w_{s_W},x_{\lceil \frac{n}{2} \rceil},x_1),$$
    which covers exactly $n$ vertices, and has colour changes at $x_{s_Z + 1}$ and either at $x_1$ or $x_{\lceil \frac{n}{2} \rceil}$ depending on the colour of $x_1x_{\lceil \frac{n}{2} \rceil}$. 
    
    So, from now on assume that there is no green or blue edge in $E(X)$, i.e.\ all edges inside $X$ are red. Consider the triples $X,Y,Z$ and $X,Y,W$. Since $$(|X| + |Y| + |Z|) + (|X| + |Y| + |W|) = N + (|X| + |Y|) \geq 2n,$$ we can conclude that one of the triples must include at least $n$ vertices. W.l.o.g., assume $|X| + |Y| + |Z| \geq n$. If $|Y| + |Z| \geq \lfloor \frac{n}{2} \rfloor$, we can find a good cycle as before by exchanging the roles of $W$ and $Y$. Otherwise, $|Y| + |Z| < \lfloor \frac{n}{2} \rfloor$. Then, we find $0 \leq s_X \leq |X|$ such that $s_X + |Y| + |Z| = n$, and a cycle
    \begin{align*}(&x_1,z_1,\ldots,x_{|Z|},z_{|Z|},x_{|Z| + 1},y_1,x_{|Z| + 2},y_2,\ldots,x_{|Z| + |Y|},y_{|Y|},\\&x_{|Z| + |Y| + 1},x_{|Z| + |Y| + 2},x_{|Z| + |Y| + 3},\ldots,x_{s_X},x_1),\end{align*}
    which covers exactly $s_X + |Y| + |Z| = n$ vertices with colour changes at $x_1$ and $x_{|Z| + 1}$.

    \textbf{Case 2:}
    Assume that $|A| < \frac{n}{2}$ for all $A \in \mathcal{S}$. Note that for distinct sets $A,B,C,D \in \mathcal{S}$ we have $|A| + |C| = N - (|B| + |D|) > 3\frac{n}{2} - 2\frac{n}{2} = \frac{n}{2}$. In particular, $|A| + |C| > |B|$ for any distinct sets $A,B,C \in \mathcal{S}$, and also $|A| = N - (|B| + |C| + |D|) > 4$ for any distinct sets $A,B,C,D \in \mathcal{S}$.

    We use Lemma~\ref{lem:sum.r.g.b} to find values $\ell^A_c \in [|A|]$ and $p^A_c \in [2]$ for every $A \in \mathcal{S}, c \in \{red,green,blue\}$ such that $A$ contains $p^A_c$ disjoint paths of colour $c$ that cover $\ell^A_c$ vertices of $|A|$ with \[\ell^A_{red} + \ell^A_{green} + \ell^A_{blue} - p^A_{red} - p^A_{green} - p^A_{blue} \geq |A| - 1\,.\] Now, for each $B \in \mathcal{S}$, and distinct $A,C \in \mathcal{S} \setminus \{B\}$, we have $|B| \geq 4\ge p^A_f + p^C_g$,
    where $f$ denotes the colour of $E(A,B)$ and 
    $g$ is the colour of $E(B,C)$. Therefore, by Lemma~\ref{lemma:ABC.paths.cycles}, there exists a good cycle $C_{B-\{A,C\}}$ on at least $\min\{|A| + |B| + |C|, 2|B| +\ell^A_f + \ell^C_g - p^A_f - p^C_g\}$ vertices using only edges from $E(A) \cup E(A,B) \cup E(B,C) \cup E(C)$. 
    
    We now want to show that one of those cycles has length at least $n$. If there is a choice $B \in \mathcal{S}$, and distinct $A,C \in \mathcal{S} \setminus \{B\}$ such that $|V(C_{B-\{A,C\}})| \geq |A| + |B| + |C| = N - |D| > n$ with $D \in \mathcal{S} \setminus \{A,B,C\}$, this holds immediately.
    
    Otherwise, assume $|V(C_{B-\{A,C\}})| \geq \ell^A_f + 2|B| + \ell^C_g - p^A_f - p^C_g$ for all $B \in \mathcal{S}$, $A,C \in \mathcal{S} \setminus \{B\}$, where again $f$ and $g$ are the colours of $E(A,B)$ and $E(B,C)$, respectively. Summing over all 12 cycles, gives
    \begin{align*}
    \sum_{\substack{B \in \mathcal{S} \\ \{A,C\} \subset \mathcal{S} \setminus \{B\}}} |V(C_{B-\{A,C\}})| & \ge 2 \sum_{A\in \mathcal{S}} (\ell^A_{red} + \ell^A_{green} + \ell^A_{blue} - p^A_{red} - p^A_{green} - p^A_{blue}) + 3\sum_{B\in \mathcal{S}} 2|B| \\
    & \geq 2 \sum_{A\in \mathcal{S}} (|A|-1) + 6 \sum_{B\in \mathcal{S}} |B|& \\
    & = 8N - 8 \geq  12n + 24,
    \end{align*}
    and we conclude that at least one of those cycles, say $C_{X-\{Y,Z\}}$, has length at least $n$.

    Finally, we show how to shorten $C_{X-\{Y,Z\}}$ to obtain a good cycle with exactly $n$ vertices. Observe that $C_{X-\{Y,Z\}}$ must use edges from both sets $E(Y) \cup E(Z)$ and $E(X,Y) \cup E(X,Z)$ because $|Y| + |Z| < n$ and $2|X| < n$. Thus, there exist three consecutive vertices $v,w,u$ on $C_{X-\{Y,Z\}}$ with $vw \in E(Y) \cup E(Z)$ and $wu \in E(X,Y) \cup E(X,Z)$. Because $vu$ and $wu$ have the same colour
    (since the only two colour changes lie in the set $X$), we can replace the sequence $v,w,u$ by $v,u$ without introducing additional colour changes. Repeating this process, we obtain a good cycle of length exactly $n$.
\end{proof}

Theorem~\ref{thm:main_even.cycle} follows from an application of Lemma~\ref{lemma:special.colouring.cycle} which leads to either a good cycle on $n$ vertices or yields a special colouring. In the latter case, Theorem~\ref{thm:cycle.split.colouring} guarantees the existence of a good cycle on $n$ vertices. 

We remark that the additional $o(n)$ term in Theorem~\ref{thm:main_even.cycle} and the restriction to even cycles only result from them being required for Lemma~\ref{lemma:special.colouring.cycle}. In contrast, Theorem~\ref{thm:cycle.split.colouring} already works for $N \geq \frac{3n}{2} + 4$ and for both odd and even cycles.

%%%%%%%%%%%%%%%%%%%%%%%%%%%%%%%%
% Concluding Remarks
%%%%%%%%%%%%%%%%%%%%%%%%%%%%%%%%
\section{Concluding Remarks}
We introduced a new Ramsey-type variant where we search for structures that are monochromatic up to a few colour changes. We provided upper and lower bounds on $R_3^1(P_n)$ that differ only by an additive constant, as well as on $R_3^2(C_n)$ for even $n$, where we require an $o(n)$ error term. 

The case of odd cycles remains open. 

\begin{problem}
  Determine the behaviour of $R_3^2(C_n)$ for odd~$n$.
\end{problem}

Our proof leads us to believe that $R_3^2(C_n)$ might have the same asymptotic behaviour for odd and even $n$, which would be very different from the behaviour of classical Ramsey numbers for cycles. In fact, the only case in our proof that does not work for odd cycles is Case~1 in the proof of Lemma~\ref{lemma:special.colouring.cycle} that leads to specific constraints on a potential extremal colouring. In particular, one would need a spanning complete bipartite graph in two colours with a cycle on $(1 + o(1))\frac{n}{2}$ vertices in the third colour in both partition classes. We were unable to construct such a colouring on more than $(\frac32 + o(1))n$ vertices that avoids a cycle on $n$ vertices with at most two colour changes for odd $n$. Extending Theorem~\ref{thm:main_even.cycle} to odd cycles would thus require a refined analysis of this case, taking into account the edges inside the partition classes of the special colouring.

Additionally, it would be desirable to prove $R_3^2(C_n) = \frac{3n}{2} + O(1)$ for even $n$. Here, the only obstruction is Case~2 in the proof of Lemma~\ref{lemma:special.colouring.cycle}, where we find two cycles on $(1 + o(1))\frac{n}{2}$ vertices with different colours. An improved analysis of this case that considers the edges within each cycle might allow the $o(n)$ error term to be replaced by $O(1)$. 

\subsection*{From path covers to paths with colour changes.}

Building on Conjecture~\ref{conj:path.cover}, we believe that even stronger statements may hold.

\begin{conjecture}\label{conj:spanning.good.path}
    For all $k,n \in \mathbb{N}$, we have $R_k^{k-1}(P_n) = n$.
\end{conjecture}

\begin{conjecture}\label{conj:spanning.good.cycle}
    For all $k,n \in \mathbb{N}$, we have $R_k^{k}(C_n) = n$.
\end{conjecture}

The first of these conjectures immediately implies Conjecture~\ref{conj:path.cover}. Indeed, given a Hamilton path with at most $k-1$ colour changes, we can delete an edge at each vertex where a colour change occurs, and are left with at most $k$ monochromatic paths covering all vertices.

\subsection*{Spanning structures.} Instead of asking for the number of vertices required to guarantee a given graph with a fixed number of colour changes, one can also ask for the number of colour changes needed to guarantee a spanning copy of a given graph. More precisely, for a graph $G$ and $q \in \mathbb{N}$, we define 
\[\kappa_q(G) := \min \big\{k \in \mathbb{N} : R_q^k(G) = |V(G)|\big\}\,.\]
This defines a monotone graph parameter, opening up a wide range of questions. A particularly interesting direction seems to be determining maximal values of $\kappa_q$ for certain families of graphs.

\begin{problem}
    Let $\mathcal{T}_n$ be the family of all trees on $n$ vertices. Determine bounds on $\max_{T \in \mathcal{T}_n} \kappa_2(T)$.
\end{problem}

Similar questions can be asked for other graph families, such as planar graphs or $k$-regular graphs.

\bibliographystyle{amsplain}
\bibliography{references}

@article{pokrovskiy2014partitioning,
  title={Partitioning edge-coloured complete graphs into monochromatic cycles and paths},
  author={Pokrovskiy, A.},
  journal={Journal of Combinatorial Theory, Series B},
  volume={106},
  pages={70--97},
  year={2014},
  publisher={Elsevier}
}

@article{benevides20093,
  title={The 3-colored {R}amsey number of even cycles},
  author={Benevides, F.S. and Skokan, J.},
  journal={Journal of Combinatorial Theory, Series B},
  volume={99},
  number={4},
  pages={690--708},
  year={2009},
  publisher={Elsevier}
}

@article {BessyThomasse,
    AUTHOR = {Bessy, S. and Thomass\'e, S.},
     TITLE = {Partitioning a graph into a cycle and an anticycle, a proof of
              {L}ehel's conjecture},
   JOURNAL = {Journal of Combinatorial Theory, Series B},
    VOLUME = {100},
      YEAR = {2010},
    NUMBER = {2},
     PAGES = {176--180},
      ISSN = {0095-8956,1096-0902},
}

@inproceedings {kohayakawa20053,
    AUTHOR = {Kohayakawa, Y. and Simonovits, M. and Skokan,
              J.},
     TITLE = {The 3-colored {R}amsey number of odd cycles},
 BOOKTITLE = {Proceedings of {GRACO} 2005},
    SERIES = { Electronic Notes in Discrete Mathematics},
    VOLUME = {19},
     PAGES = {397--402},
 PUBLISHER = {Elsevier, Amsterdam},
      YEAR = {2005},
       DOI = {10.1016/j.endm.2005.05.053},
       URL = {https://doi.org/10.1016/j.endm.2005.05.053},
}

@article{moon1963hamiltonian,
  title={On {H}amiltonian bipartite graphs},
  author={Moon, J. and Moser, L.},
  journal={Israel Journal of Mathematics},
  volume={1},
  number={3},
  pages={163--165},
  year={1963},
  publisher={Springer New York}
}

@article{erdos1935combinatorial,
  title={A combinatorial problem in geometry},
  author={Erd{\H{o}}s, P. and Szekeres, G.},
  journal={Compositio Mathematica},
  volume={2},
  pages={463--470},
  year={1935}
}

@article{erdos1947some,
  title={Some remarks on the theory of graphs},
  author={Erd{\H{o}}s, P.},
  journal={Bulletin of the American Mathematical Society},
  volume={53},
  number={4},
  pages={292--294},
  year={1947}
}

@article{campos2023exponential,
  title={An exponential improvement for diagonal {R}amsey},
  author={Campos, M. and Griffiths, S. and Morris, R. and Sahasrabudhe, J.},
FJOURNAL = {Ann. of Math. (2)},
  JOURNAL = {Annals of Mathematics. Second Series},
    VOLUME = {203},
      YEAR = {2026},
    NUMBER = {3},
     PAGES = {869--932},
       DOI = {10.4007/annals.2026.203.3.4},
       URL = {https://doi.org/10.4007/annals.2026.203.3.4},
}

@article{ramsey1930problem,
  title={On a Problem of Formal Logic},
  author={Ramsey, F.P.},
  journal={Proceedings of the London Mathematical Society},
  volume={2},
  number={1},
  pages={264--286},
  year={1930},
  publisher={Wiley Online Library}
}

@article{chvatal1983ramsey,
  title={The {R}amsey number of a graph with bounded maximum degree},
  author={Chv{\'a}tal, V. and R{\"o}dl, V. and Szemer{\'e}di, E. and Trotter Jr, W.T.},
  journal={Journal of Combinatorial Theory, Series B},
  volume={34},
  number={3},
  pages={239--243},
  year={1983},
  publisher={Elsevier}
}

@incollection{BurrErdos1975,
  author       = {Burr, S.A. and Erd{\H{o}}s, P.},
  title        = {On the magnitude of generalized {R}amsey numbers for graphs},
  booktitle    = {Infinite and Finite Sets, Vol.\ 1, Colloquia Mathematica Societatis János Bolyai, 10},
  pages        = {214--240},
  publisher    = {North-Holland},
  address      = {Amsterdam / London},
  year         = {1975}
}

@article{gerencser1967ramsey,
    AUTHOR = {Gerencs\'er, L. and Gy\'arf\'as, A.},
     TITLE = {On {R}amsey-type problems},
  JOURNAL = {Annales Universitatis Scientiarum Budapestinensis de Rolando
              E\"otv\"os Nominatae. Sectio Mathematica},
    VOLUME = {10},
      YEAR = {1967},
     PAGES = {167--170},
}

@article{gyarfas2007three,
  title={Three-color {R}amsey numbers for paths},
  author={Gy{\'a}rf{\'a}s, A. and Ruszink{\'o}, M. and S{\'a}rk{\"o}zy, G. and Szemer{\'e}di, E.},
  journal={Combinatorica},
  volume={27},
  number={1},
  pages={35--70},
  year={2007},
  publisher={Budapest: Akademiai Kiado,[1981-}
}

@article{knierim2019improved,
  title={Improved bounds on the multicolor {R}amsey numbers of paths and even cycles},
  author={Knierim, C. and Su, P.},
  journal={The Electronic Journal of Combinatorics},
  volume={26},
  number={1},
  pages={1--26},
  year={2019},
  publisher={Electronic Journal of Combinatorics}
}

@article{yongqi2006new,
  title={New lower bounds on the multicolor {R}amsey numbers {$R_r(C_{2m})$}},
  author={Yongqi, S. and Yuansheng, Y. and Feng, X. and Bingxi, L.},
  journal={Graphs and Combinatorics},
  volume={22},
  number={2},
  pages={283--288},
  year={2006},
  publisher={Springer}
}

@book{west2001introduction,
  title={Introduction to graph theory},
  author={West, D.B.},
 PUBLISHER = {Prentice Hall, Inc., Upper Saddle River, NJ},
      YEAR = {2001},
      edition={2nd},
     PAGES = {588+xx},
      ISBN = {0-13-014400-2},
}

@article{rosta1973ramsey,
  title={On a {R}amsey-type problem of {JA B}ondy and {P}. {E}rd{\"o}s. {I}},
  author={Rosta, V.},
  journal={Journal of Combinatorial Theory, Series B},
  volume={15},
  number={1},
  pages={94--104},
  year={1973},
  publisher={Elsevier}
}

@article{bialostocki1991simple,
  title={On simple Hamiltonian cycles in a 2-colored complete graph},
  author={Bialostocki, A. and Dierker, P.},
  journal={Ars Combinatoria},
  volume={32},
  pages={13--16},
  year={1991}
}

@article{bondy1973ramsey,
  title={{R}amsey numbers for cycles in graphs},
  author={Bondy, J.A. and Erd{\H{o}}s, P.},
  journal={Journal of Combinatorial Theory, Series B},
  volume={14},
  number={1},
  pages={46--54},
  year={1973},
  publisher={Elsevier}
}

@article{faudree1974all,
  title={All {R}amsey numbers for cycles in graphs},
  author={Faudree, R.J. and Schelp, R.H.},
  journal={Discrete Mathematics},
  volume={8},
  number={4},
  pages={313--329},
  year={1974},
  publisher={Elsevier}
}

@incollection {gyarfas1989covering,
    AUTHOR = {Gy{\'a}rf{\'a}s, A.},
     TITLE = {Covering complete graphs by monochromatic paths},
 BOOKTITLE = {Irregularities of partitions ({F}ert\H od, 1986)},
    SERIES = {Algorithms and Combinatorics: Study and Research Texts},
    VOLUME = {8},
     PAGES = {89--91},
 PUBLISHER = {Springer, Berlin},
      YEAR = {1989},
}

@article{erdHos1991vertex,
  title={Vertex coverings by monochromatic cycles and trees},
  author={Erd{\H{o}}s, P. and Gy{\'a}rf{\'a}s, A. and Pyber, L.},
  journal={Journal of Combinatorial Theory, Series B},
  volume={51},
  number={1},
  pages={90--95},
  year={1991},
  publisher={Elsevier}
}

@article{luczak1998partitioning,
  title={Partitioning two-coloured complete graphs into two monochromatic cycles},
  author={{\L}uczak, T. and R{\"o}dl, V. and Szemer{\'e}di, E.},
  journal={Combinatorics, Probability and Computing},
  volume={7},
  number={4},
  pages={423--436},
  year={1998},
  publisher={Cambridge University Press}
}

@article{ALLEN_2008, 
title={Covering Two-Edge-Coloured Complete Graphs with Two Disjoint Monochromatic Cycles}, 
volume={17}, 
DOI={10.1017/S0963548308009164}, 
number={4}, 
journal={Combinatorics, Probability and Computing}, 
author={Allen, P.}, 
year={2008}, 
pages={471–486}}

@article{feder2013hypercube,
  title={On hypercube labellings and antipodal monochromatic paths},
  author={Feder, T. and Subi, C.},
  journal={Discrete Applied Mathematics},
  volume={161},
  number={10-11},
  pages={1421--1426},
  year={2013},
  publisher={Elsevier}
}

@article{raynaud1973circuit,
  title={Sur le circuit hamiltonien bi-color{\'e} dans les graphes orient{\'e}s},
  author={Raynaud, H.},
  journal={Periodica Mathematica Hungarica},
  volume={3},
  number={3-4},
  pages={289--297},
  year={1973},
  publisher={Akad{\'e}miai Kiad{\'o}}
}

@article{gyarfas1983vertex,
  title={Vertex coverings by monochromatic paths and cycles},
  author={Gy{\'a}rf{\'a}s, A.},
  journal={Journal of Graph Theory},
  volume={7},
  number={1},
  pages={131--135},
  year={1983},
  publisher={Wiley Online Library}
}

@article{manoussakis1996cycles,
    AUTHOR = {Manoussakis, Y. and Spyratos, M. and Tuza, Zs.},
     TITLE = {Cycles of given color patterns},
   JOURNAL = {Journal of Graph Theory},
    VOLUME = {21},
      YEAR = {1996},
    NUMBER = {2},
     PAGES = {153--162},
       DOI = {10.1002/(SICI)1097-0118(199602)21:2<153::AID-JGT4>3.3.CO;2-R},
}

@article{leader2014long,
  title={Long geodesics in subgraphs of the cube},
  author={Leader, I. and Long, E.},
  journal={Discrete Mathematics},
  volume={326},
  pages={29--33},
  year={2014},
  publisher={Elsevier}
}

@misc{NorineEdgeAntipodal,
  author       = {Norine, S.},
  title        = {Edge-antipodal colorings of cubes},
  year         = {2008},
  note          = {\url{http://www.openproblemgarden.org/op/edge_antipodal_colorings_of_cubes}},
  urldate      = {2025},
  organization = {Open Problem Garden}
}

@article{dvovrak2020note,
  title={A Note on {N}orine's Antipodal-Colouring Conjecture},
  author={Dvo{\v{r}}{\'a}k, V.},
  journal={The Electronic Journal of Combinatorics},
  pages={2--26},
  year={2020}
}

@article{gyarfas1973ramsey,
  title={A {R}amsey-type problem in directed and bipartite graphs},
  author={Gy{\'a}rf{\'a}s, A. and Lehel, J.},
  journal={Periodica Mathematica Hungarica},
  volume={3},
  number={3-4},
  pages={299--304},
  year={1973},
  publisher={Akad{\'e}miai Kiad{\'o}}
}

@unpublished{gupta2024optimizing,
  title={Optimizing the {CGMS} upper bound on {R}amsey numbers},
  author={Gupta, P. and Ndiaye, N. and Norin, S. and Wei, L.},
  note={arXiv:2407.19026},
  year={2024}
}

@article {balister2024upper,
    AUTHOR = {Balister, P. and Bollob\'as, B. and Campos, M. and
              Griffiths, S. and Hurley, E. and Morris, R. and
              Sahasrabudhe, J. and Tiba, M.},
     TITLE = {Upper bounds for multicolour {R}amsey numbers},
  JOURNAL = {Journal of the American Mathematical Society},
    VOLUME = {39},
      YEAR = {2026},
    NUMBER = {3},
     PAGES = {765--780},
       DOI = {10.1090/jams/1069},
       URL = {https://doi.org/10.1090/jams/1069},
}

@article {Lee_BurrErdos,
    AUTHOR = {Lee, C.},
     TITLE = {Ramsey numbers of degenerate graphs},
  JOURNAL = {Annals of Mathematics},
    VOLUME = {185},
      YEAR = {2017},
    NUMBER = {3},
     PAGES = {791--829},
       DOI = {10.4007/annals.2017.185.3.2},
       URL = {https://doi.org/10.4007/annals.2017.185.3.2},
}

\end{document}